\documentclass[12pt]{amsart}

\usepackage{latexsym,amssymb} 
\usepackage{amsmath,amsfonts,amsthm}
\input xypic

\newtheorem{theorem}{Theorem}[section]
\newtheorem{proposition}[theorem]{Proposition}
\newtheorem{corollary}[theorem]{Corollary}
\newtheorem{lemma}[theorem]{Lemma}

\newtheorem{remarks}[theorem]{Remarks}

\begin{document}

\title{Morita equivalence of inverse semigroups}

\author{B.~Afara}
\address{Department of Mathematics\\
Heriot-Watt University\\
Riccarton\\
Edinburgh~EH14~4AS\\
Scotland}
\email{B.Afara@ma.hw.ac.uk}
\thanks{The first author would like to thank Aleppo University, Syria and the British Council for their support.}

\author{M.~V.~Lawson}
\address{Department of Mathematics and the Maxwell Institute for Mathematical Sciences\\
Heriot-Watt University\\
Riccarton\\
Edinburgh~EH14~4AS\\
Scotland}
\email{M.V.Lawson@ma.hw.ac.uk}
\thanks{The second author would like to thank Prof Gracinda Gomes and the CAUL project ISFL-1-143 supported by FCT }



\begin{abstract}
We describe how to construct all inverse semigroups Morita equivalent to a given inverse semigroup $S$.
This is done by taking the maximum inverse images of the regular Rees matrix semigroups
over $S$ where the sandwich matrix satisfies what we call the McAlister conditions.
\end{abstract}

\maketitle
\section{Introduction}

The Morita theory of monoids was introduced independently by Banaschewski \cite{Ban} and Knauer \cite{K} as the analogue of the classical Morita theory of rings \cite{Lam}.
This theory was extended to semigroups with local units by  Talwar \cite{T1,T2,T3};
a semigroup $S$ is said to have {\em local units} if for each $s \in S$ there exist idempotents $e$ and $f$ such that $s = esf$.
Inverse semigroups have local units and the definition of Morita equivalence in their case assumes the following form.
Let $S$ be an inverse semigroup.
If $S$ acts on a set $X$ in such a way that $SX = X$ we say that the action is {\em unitary}.
We denote by $S$-\mbox{\bf mod} the category of unitary left $S$-sets and their  left $S$-homomorphisms.
Inverse semigroups $S$ and $T$ are said to be {\em Morita equivalent} if the categories  $S$-\mbox{\bf mod} and  $T$-\mbox{\bf mod} are equivalent.
There have been a number of recent papers on this topic \cite{FLS,L3,L4,S} and ours takes the development of this theory a stage further.

Rather than taking the definition of Morita equivalence as our starting point, we shall use instead two characterizations that are much easier to work with.
We denote by $C(S)$ the {\em Cauchy completion} of the semigroup $S$.
This is the category with elements triples of the form $(e,s,f)$, 
where $s = esf$ and $e$ and $f$ are idempotents, and multiplication given by $(e,s,f)(f,t,g) = (e,st,g)$.

The first characterization is the following \cite{FLS}.

\begin{theorem} Let $S$ and $T$ be semigroups with local units.
Then $S$ and $T$ are Morita equivalent if and only if their Cauchy completions are equivalent.
\end{theorem}

To describe the second characterization we shall need the following definition from \cite{S}.
Let $S$ and $T$ be inverse semigroups.
An {\em equivalence biset from $S$ to $T$} consists of an $(S,T)$-biset $X$ equipped with surjective functions
$$
\langle -,- \rangle \colon \: X \times X \rightarrow S\;,
\text{ and }
[-,-] \colon \:X \times X \rightarrow T
$$
such that the following axioms hold,
where $x,y,z \in X$, $s \in S$, and $t \in T$:
\begin{description}

\item[{\rm (M1)}] $\langle sx,y \rangle = s\langle x,y \rangle$

\item[{\rm (M2)}] $\langle y,x \rangle = \langle x,y \rangle^{-1}$

\item[{\rm (M3)}] $\langle x,x \rangle x = x$

\item[{\rm (M4)}] $[x,yt] = [x,y]t$

\item[{\rm (M5)}] $[x,y] = [y,x]^{-1}$

\item[{\rm (M6)}] $x[x,x] = x$

\item[{\rm (M7)}] $\langle x,y \rangle z = x [y,z]$.

\end{description}

Observe that by (M6) and (M7), we have that
$\langle x, x \rangle x = x [x,x] = x$.

Recall that a {\em weak equivalence} from one category to another is a functor that is full, faithful and essentially surjective.
By the Axiom of Choice, categories are equivalent if and only if there is a weak equivalence between them.
It is not hard to see, Theorem~5.1 of \cite{S}, that if there is an equivalence biset from $S$ to $T$ then there is a weak equivalence
from $C(S)$ to $C(T)$ and so by Theorem~1.1, the inverse semigroups $S$ and $T$ are Morita equivalent.
In fact, the converse is true by Theorem~2.14 of \cite{FLS}.

\begin{theorem} Let $S$ and $T$ be inverse semigroups.
Then $S$ and $T$ are Morita equivalent if and only if there is an equivalence biset from $S$ to $T$.
\end{theorem}

The goal of this paper can now be stated: given an inverse semigroup $S$ how do we construct all inverse semigroups $T$ that
are Morita equivalent to $S$?
We shall show how to do this.
This paper can be seen as a generalization and completion of some of the results to be found in \cite{L1}.

Our main reference for general semigroup theory is Howie \cite{H} and for inverse semigroups Lawson \cite{MVL}.
Since categories play a role, it is worth stressing, to avoid confusion, that a semigroup $S$ is {\em (von Neumann) regular}
if each element $s \in S$ has an {\em inverse} $t$ such that $s = sts$ and $t = tst$.
The set of inverses of $s$ is denoted by $V(s)$.
Inverse semigroups are the regular semigroups in which each element has a unique inverse.

\section{The main construction}

Our main tool will be Rees matrix semigroups.
These can be viewed as the semigroup analogues of matrix rings and, the reader will recall,
matrix rings play an important role in the Morita theory of unital rings \cite{Lam}.

If $S$ is a regular semigroup then a Rees matrix semigroup $M(S;I,\Lambda;P)$ over $S$ need not be regular.
However, we do have the following.

\begin{lemma}[Lemma~2.1 of \cite{M1}] Let $S$ be a regular semigroup.
Let  $RM(S;I,\Lambda;P)$ be the set of regular elements of  $M(S;I,\Lambda;P)$.
Then $RM(S;I,\Lambda;P)$ is a regular semigroup.
\end{lemma}

The semigroup $RM(S;I,\Lambda;P)$ is called a {\em regular Rees matrix semigroup} over $S$.
Recall that a {\em local submonoid} of a semigroup $S$ is a subsemigroup of the form $eSe$ where $e$ is an idempotent.
A regular semigroup $S$ is said to be {\em locally inverse} if each local submonoid is inverse.
Regular Rees matrix semigroups over inverse semigroups need not be inverse, but we do have the following.
The proof follows by showing that each local submonoid of $RM(S;I,\Lambda;P)$ is isomorphic to a local submonoid of $S$.

\begin{lemma}[Lemma~1.1 of \cite{M2}] Let $S$ be an inverse semigroup.
Then a regular Rees matrix semigroup over $S$ is locally inverse.
\end{lemma}

Regular Rees matrix semigroups over inverse semigroups are locally inverse but not inverse.
To get closer to being an inverse semigroup we need to impose more conditions on the Rees matrix semigroup.
First, we shall restrict our attention to {\em square} Rees matrix semigroups: those semigroups where $I = \Lambda$.
In this case, we shall denote our Rees matrix semigroup by $M(S,I,p)$ where
$p \colon I \times I \rightarrow S$ is the function giving the entries of the sandwich matrix $P$.
Next, we shall place some conditions on the sandwich matrix $P$:

\begin{description}

\item[{\rm (MF1)}] $p_{i,i}$ is an idempotent for all $i \in I$.

\item[{\rm (MF2)}] $p_{i,i}p_{i,j}p_{j,j} = p_{i,j}$.

\item[{\rm (MF3)}] $p_{i,j} = p_{j,i}^{-1}$.

\item[{\rm (MF4)}] $p_{i,j}p_{j,k} \leq p_{i,k}$.

\item[{\rm (MF5)}] For each $e \in E(S)$ there exists $i \in I$ such that $e \leq p_{i,i}$.

\end{description}
We shall call functions satisfying all these conditions {\em McAlister functions}.
Our choice of name reflects the fact that McAlister was the first to study functions of this kind in \cite{M2}.

The following is essentially Theorem~6.7 of \cite{L1} but we include a full proof for the sake of completeness.

\begin{lemma} Let  $M = M(S,I,p)$ where $p$ satisfies (M1)--(M4).

\begin{enumerate}

\item $(i,s,j)$ is regular if and only if $s^{-1}s \leq p_{j,j}$ and $ss^{-1} \leq p_{i,i}$.

\item If $(i,s,j)$ is regular then one of its inverses is $(j,s^{-1},i)$.

\item $(i,s,j)$ is an idempotent if and only if $s \leq p_{i,j}$.
 
\item The idempotents form a subsemigroup.

\end{enumerate}
\end{lemma}
\begin{proof}

(1). Suppose that $(i,s,j)$ is a regular element.
Then there is an element $(k,t,l)$ such that
$(i,s,j) = (i,s,j)(k,t,l)(i,s,j)$ and $(k,t,l) = (k,t,l)(i,s,j)(k,t,l)$.
Thus, in particular, $s = sp_{j,k}tp_{l,i}s$. 
Now
$$p_{j,j}s^{-1}s = p_{j,j}s^{-1}sp_{j,k}tp_{l,i}s = s^{-1}s p_{j,j}p_{j,k}tp_{l,i}s$$
using the fact that $p_{j,j}$ is an idempotent.
But $p_{j,j}p_{j,k} = p_{j,k}$ and so 
$$p_{j,j}s^{-1}s = s^{-1}s p_{j,k}tp_{l,i}s = s^{-1}s.$$
Thus $s^{-1}s \leq p_{j,j}$.
By symmetry, $ss^{-1} \leq p_{i,i}$.

(2) This is a straightforward verification.

(3). Suppose that $(i,s,j)$ is an idempotent.
Then $s = sp_{j,i}s$.
It follows that $s^{-1} = s^{-1}s p_{j,i} ss^{-1} \leq p_{j,i}$ and so $s \leq p_{i,j}$.
Conversely, suppose that $s \leq p_{i,j}$.
Then $s^{-1} \leq p_{j,i}$ and so $s^{-1} = s^{-1}s p_{j,i}ss^{-1}$ which gives $s = sp_{j,i}s$.
This implies that $(i,s,j)$ is an idempotent.

(4). Let $(i,s,j)$ and $(k,t,l)$ be idempotents.
Then by (2) above we have that $s \leq p_{i,j}$ and $t \leq p_{k,l}$.
Now $(i,s,j)(k,t,l) = (i,sp_{j,k}t,l)$.
But $sp_{j,k}t \leq p_{i,j}p_{j,k}p_{k,l} \leq p_{i,l}$.
It follows that $(i,s,j)(k,t,l)$ is an idempotent.
\end{proof}

A regular semigroup is said to be {\em orthodox} if its idempotents form a subsemigroup.
Inverse semigroups are orthodox.
An orthodox locally inverse semigroup is called a {\em generalized inverse semigroup}.
They are the orthodox semigroups whose idempotents form a normal band.

\begin{corollary} Let $S$ be an inverse semigroup.
If $M = M(S,I,p)$ where $p$ satisfies (M1)--(M4) then $RM(S,I,p)$ is a generalized inverse semigroup.
\end{corollary}

Let $S$ be a regular semigroup.
Then the intersection of all congruences $\rho$ on $S$ such $S/\rho$ is inverse is a congruence denoted by $\gamma$;
it is called the {\em minimum inverse congruence}.

\begin{lemma}[Theorems~6.2.4 and 6.2.5 of \cite{H}] Let $S$ be an orthodox semigroup.
Then the following are equivalent:
\begin{enumerate}

\item $s \, \gamma \, t$.

\item $V(s) \cap V(t) \neq \emptyset$.

\item $V(s) = V(t)$. 

\end{enumerate}
\end{lemma}

\begin{lemma} Let  $RM = RM(S,I,p)$ where $p$ satisfies (M1)--(M4).
Then $(i,s,j) \gamma (k,t,l)$ if and only if $s = p_{i,k}tp_{l,j}$ and $t = p_{k,l}sp_{j,l}$.
\end{lemma}
\begin{proof} Lemma~2.5 forms the backdrop to this proof.
Suppose that $(i,s,j) \gamma (k,t,l)$.
Then the two elements have the same sets of inverses.
Now $(j,s^{-1},i)$ is an inverse of $(i,s,j)$ and so by assumption it is an inverse of $(k,t,l)$.
Thus 
$$t = tp_{l,j}s^{-1}p_{i,k}t \text{ and } s^{-1} = s^{-1}p_{i,k}tp_{l,j}s^{-1}.$$
It follows that
$$s \leq p_{i,k}tp_{l,j} \text{ and }t^{-1} \leq p_{l,j}s^{-1}p_{i,k}$$
so that
$$t \leq p_{k,i}sp_{j,l}.$$
Now
$$s \leq  p_{i,k}tp_{l,j} \leq p_{i,k}p_{k,i}sp_{j,l}p_{l,j} \leq p_{i,i}sp_{j,j} = s.$$
Thus $s = p_{i,k}tp_{l,j}$.
Similarly, $t = p_{k,i}sp_{j,l}$.

Conversely, suppose that  $s = p_{i,k}tp_{l,j}$ and $t = p_{k,l}sp_{j,l}$.
We shall prove that $V(i,s,j) \cap V(k,t,l) \neq \emptyset$.
To do this, we shall prove that $(j,s^{-1},i)$ is an inverse of $(k,t,l)$.
We calculate
$$tp_{l,j}s^{-1}p_{i,k}t = t(p_{i,j}s^{-1}p_{i,k})t = t(p_{k,i}sp_{j,l})^{-1}t = tt^{-1}t = t.$$
Similarly, $s^{-1} = s^{-1}p_{i,k}tp_{l,j}s^{-1}$.
The result now follows.\end{proof}

With the assumptions of the above lemma, put 
$$IM(S,I,p) = RM(S,I,p)/\gamma.$$
We call $IM(S,I,p)$ the {\em inverse Rees matrix semigroup} over $S$.

A homomorphism $\theta \colon S \rightarrow T$ between semigroups with local units is said to be a {\em local isomorphism}
if the following two conditions are satisfied:
\begin{description}

\item[{\rm (LI1)}] $\theta \mid eSf \colon eSf \rightarrow \theta (e) T \theta (f)$ is an isomorphism for all idempotents $e,f \in S$.

\item[{\rm (LI2)}] For each idempotent $i \in T$ there exists an idempotent $e \in S$ such that $i \mathcal{D} \theta (e)$.

\end{description}
This definition is a slight refinement of the one given in \cite{L3}.

\begin{lemma} Let $\theta \colon S \rightarrow T$ be a surjective homomorphism between regular semigroups.
Then $\theta$ is a local isomorphism if and only if 
$\theta \mid eSe \colon eSe \rightarrow \theta (e) T \theta (e)$ is an isomorphism for all idempotents $e \in S$.
\end{lemma}
\begin{proof} The homomorphism is surjective and so (LI2) is automatic.
We need only prove that (LI2) follows from the assumption that
$\theta \mid eSe \colon eSe \rightarrow \theta (e) T \theta (e)$ is an isomorphism for all idempotents $e \in S$.
This follows from Lemma~1.3 of \cite{M2}.
\end{proof}

\begin{lemma}[Proposition~1.4 of \cite{M2}] Let $S$ be a regular semigroup.
Then the natural homomorphism from $S$ to $S/\gamma$ is a local isomorphism if and only if $S$ is a generalized inverse semigroup. 
\end{lemma}

Our next two results bring Morita equivalence into the picture via Theorem~1.1.

\begin{lemma} Let $S$ and $T$ be inverse semigroups.
If $\theta \colon S \rightarrow T$ is a surjective local isomorphism then $S$ and $T$ are Morita equivalent.
\end{lemma}
\begin{proof} Define $\Theta \colon C(S) \rightarrow C(T)$ by $\Theta (e,s,f) = (\theta (e), \theta (s), \theta (f))$.
Then $\Theta$ is a functor, and it is full and faithful because $\theta$ is a local isomorphism.
Identities in $C(T)$ have the form $(i,i,i)$ where $i$ is an idempotent in $T$.
Because $\theta$ is surjective and $S$ is inverse there is an idempotent $e \in S$ such that $\theta (e) = i$.
Thus every identity in $C(T)$ is the image of an identity in $C(S)$.
It follows that $\Theta$ is a weak equivalence.
Thus the categories $C(S)$ and $C(T)$ are equivalent and so, by Theorem~1.1, the semigroups $S$ and $T$ are Morita equivalent.
\end{proof}

\begin{lemma} Let  $M = M(S,I,p)$ where $p$ satisfies (MF1)--(MF5).
Then $S$ is Morita equivalent to $RM(S,I,p)$.
\end{lemma}
\begin{proof} We shall construct a weak equivalence from $C(RM(S,I,p))$ to $C(S)$.
By Theorem~1.1 this implies that $S$ is Morita equivalent to $RM(S,I,p)$.
A typical element of $C(RM(S,I,p))$ has the form
$$\mathbf{s} = [(i,a,j),(i,s,k),(l,b,k)]$$
where $(i,s,j)$ is regular and $(i,a,j)$ and $(l,b,k)$ are idempotents and $(i,a,j)(i,s,k)(l,b,k) = (i,s,k)$.
Observe that both $ap_{j,i}$ and $bp_{k,l}$ are idempotents and that $(ap_{j,i})sp_{k,l}(bp_{k,l}) = sp_{k,l}$.
It follows that 
$$(ap_{j,i},sp_{k,l},bp_{k,l})$$
is a well-defined element of $C(S)$.
We may therefore define
$$\Psi \colon C(RM(S,I,p)) \rightarrow C(S)$$
by
$$\Psi[(i,a,j),(i,s,k),(l,b,k)] = (ap_{j,i},sp_{k,l},bp_{k,l}).$$
It is now easy to check that $\Psi$ is full and faithful.
Let $(e,e,e)$ be an arbitrary identity of $C(S)$.
Then $e$ is an idempotent in $S$.
By (MF5), there exists $i \in I$ such that $e \leq p_{i,i}$.
It follows that $(i,e,i)$ is an idempotent in $RM(S,I,p)$.
Thus
$$[(i,e,i),(i,e,i),(i,e,i)]$$
is an identity in $C(RM(S,I,p))$.
But 
$$\Psi [(i,e,i),(i,e,i),(i,e,i)] = (ep_{i,i},ep_{i,i},ep_{i,i}) = (e,e,e).$$
Thus every identity in $C(S)$ is the image under $\Psi$ of an identity in  $C(RM(S,I,p))$.
In particular, $\Psi$ is essentially surjective.\end{proof}

We may summarize what we have found so far in the following result.

\begin{proposition} Let $S$ be an inverse semigroup and let $p \colon I \times I \rightarrow S$ be a McAlister function.
Then $S$ is Morita equivalent to the inverse Rees matrix semigroup $IM(S,I,p)$. 
\end{proposition}

\section{The main theorem}

Our goal now is to prove that all inverse semigroups Morita equivalent to $S$ are isomorphic to inverse Rees matrix semigroups $IM(S,I,p)$.
We shall use Theorem~1.2.
We begin with some results about equivalence bisets all of which are taken from \cite{S}.

The following is part of Proposition~2.3 \cite{S}.

\begin{lemma} Let $(S,T,X,\langle -,- \rangle,[-,-])$ be an equivalence biset.
\begin{enumerate}

\item For each $x \in X$ both $\langle x, x \rangle$ and $[x,x]$ are idempotents.

\item $\langle x, y \rangle \langle z, w \rangle = \langle x[y,z], w \rangle$.

\item $[x,y][z,w] = [x, \langle y, z \rangle w]$.

\item $\langle xt, y\rangle = \langle x, yt^{-1}\rangle$.

\item $[sx,y] = [x,s^{-1}y]$.

\end{enumerate}
\end{lemma}

\begin{lemma} Let $(S,T,X,\langle,\rangle,[,])$ be an equivalence biset from $S$ to $T$.
\begin{enumerate}

\item For each $x \in X$ there exists a homomorphism $\epsilon_{x} \colon E(S) \rightarrow E(T)$ such that
$ex = x\epsilon_{x}(e)$ for all $e \in E(S)$.

\item  For each $x \in X$ there exists a homomorphism $\eta_{x} \colon E(S) \rightarrow E(T)$ such that
$xf = \eta_{x}(f)x$ for all $e \in E(S)$.

\end{enumerate}
\end{lemma}
\begin{proof} We prove (1); the proof of (2) follows by symmetry.
Define $\epsilon_{x}$ by $\epsilon_{x}(e) = [ex,ex]$.
By Proposition~2.4 of \cite{S}, this is a semigroup homomorphism.
Next we use the argument from Proposition~3.6 of \cite{S}.
We calculate $x[ex,ex]$ as follows
$$x[ex,ex] = \langle x, ex \rangle ex = \langle x, x \rangle ex = e \langle x, x \rangle x = ex,$$ 
as required.
\end{proof}

\begin{lemma} Let $(S,T,X,\langle,\rangle,[,])$ be an equivalence biset from $S$ to $T$.
Define $p \colon X \times X \rightarrow S$ by $p_{x,y} = \langle x, y \rangle$.
Then $p$ is a McAlister function.
\end{lemma}
\begin{proof}

(MF1) holds. By Lemma~3.1(1), $\langle x, x \rangle$ is an idempotent.

(MF2) holds. By Lemma~3.1(2), $\langle x, x \rangle \langle x, y \rangle = \langle x[x,x],y \rangle$.
But $x[x,x] = x$ by (M6), and so  $\langle x, x \rangle \langle x, y \rangle = \langle x,y \rangle$.
The other result holds dually.

(MF3) holds. This follows from (M2).

(MF4) holds. By Lemma~3.1(2), we have that
$\langle x, y \rangle \langle y, z \rangle = \langle x[y,y], z \rangle$.
By Lemma~3.2, we have that $x[y,y] = \eta_{x}([y,y])x = fx$.
Thus
$\langle x[y,y], z \rangle = \langle fx, x \rangle = f \langle x, z \rangle \leq \langle x, z \rangle$.

(MF5) holds. Let $e \in E(S)$.
Then since $\langle -,-\rangle$ is surjective, there exists $x,y \in X$ such that $e = \langle x, y \rangle$.
But then $e = \langle x, y \rangle\langle y, x \rangle \leq \langle x, x \rangle = p_{x,x}$.\end{proof}

\begin{lemma} Let $(S,T,X,\langle,\rangle,[,])$ be an equivalence biset from $S$ to $T$.
Define $p \colon X \times X \rightarrow S$ by $p_{x,y} = \langle x, y \rangle$.
Form the regular Rees matrix semigroup $R = RM(S,X,p)$.
Define $\theta \colon  RM(S,X,p) \rightarrow T$ by $\theta (x,s,y) = [x,sy]$.
Then $\theta$ is a surjective homomorphism with kernel $\gamma$.
\end{lemma}
\begin{proof} We show first that $\theta$ is a homomorphism.
By definition 
$$(x,s,y)(u,t,v) = (x,s\langle y,u\rangle t,v).$$
Thus 
$$\theta ((x,s,y)(u,t,v)) = [x, s \langle y,u \rangle tv],$$ 
whereas
$$\theta (x,s,y)\theta (u,t,v) = [x,sy][u,tv].$$
By Lemma~3.1(3), we have that
$$[x,sy][u,tv] = [x, \langle sy,u \rangle tv]$$
but by (M1), $\langle sy,u \rangle = s \langle y,u \rangle$.
It follows that $\theta$ is a homomorphism.

Next we show that $\theta$ is surjective.
Let $t \in T$.
Then there exists $(x,y) \in X \times X$ such that $[x,y] = t$.
Consider the element $(x,\langle x, x \rangle \langle y, y \rangle,y)$ of $M(S,I,p)$.
This is in fact an element of $RM(S,X,p)$.
The image of this element under $\theta$ is 
$$[x, \langle x,x \rangle \langle y, y \rangle y] = [x, \langle x,x \rangle y]$$
since $\langle y, y \rangle y = y$.
But by Lemma~3.1(5), we have that
$$[x, \langle x,x \rangle y] = [\langle x,x \rangle x, y] = [x,y] = t,$$ 
as required.

It remains to show that the kernel of $\theta$ is $\gamma$.
Let $(x,s,y),(u,t,v) \in RM(S,X,p)$.
Suppose first that $\theta (x,s,y) = \theta (u,t,v)$.
By definition, $[x,sy] = [u,tv]$.
Then
$$s 
= \langle x, x \rangle s \langle y, y \rangle 
=   \langle x, x \rangle \langle sy, y \rangle
= \langle x[x,sy], y \rangle$$
by Lemma~3.1(2).
But $[x,sy] = [u,tv]$.
Thus 
$$s 
=  \langle x[u,tv], y \rangle
=
\langle x, u \rangle \langle tv,y\rangle = \langle x, u \rangle t \langle v, y \rangle.$$
By symmetry and Lemma~2.6, we deduce that $(x,s,y) \gamma (u,t,v)$.

Suppose now that $(x,s,y) \gamma (u,t,v)$.
Then by Lemma~2.6
$$s = \langle x, u \rangle t \langle v, y \rangle 
\text{ and }
t = \langle u, x \rangle s \langle y, v \rangle.$$ 
Now
$$[x,sy] 
= [x,\langle x, u \rangle t \langle v, y \rangle y]
= [x,\langle x, u \rangle tv [y, y]]
= [u[x,x],tv[y,y]] = [x,x][u,tv][y,y]$$
using Lemma~3.1.
This gives $[x,sy] \leq [u,tv]$. 
A symmetric argument shows that $[u,tv] \leq [x,sy]$.
Hence $[x,sy] = [u,tv]$, as required.
\end{proof}

We may now state our main theorem.

\begin{theorem} Let $S$ be an inverse semigroup.
For each McAlister function $p \colon I \times I \rightarrow S$ the inverse Rees matrix semigroup $IM(S,I,p)$ is Morita equivalent to $S$, 
and every inverse semigroup Morita equivalent to $S$ is isomorphic to one of this form.
\end{theorem}

\begin{remarks} \mbox{}

\begin{enumerate}

\item {\em Let $S$ be an inverse monoid and suppose that $p \colon I \times I \rightarrow S$ is a function satisfying (MF1)--(MF5).
Condition (MF5) says that For each $e \in E(S)$ there exists $i \in I$ such that $e \leq p_{i,i}$.
Thus,in particular, there exists $i_{0} \in I$ such that $1 \leq p_{i_{0},i_{0}}$.
But $p_{i_{0},i_{0}}$ is an idempotent and so $1 = p_{i_{0},i_{0}}$.
Suppose now that  $p \colon I \times I \rightarrow S$ is a function satisfying (MF1)--(MF4)
and there exists $i_{0} \in I$ such that  $1 = p_{i_{0},i_{0}}$.
Every idempotent $e \in S$ satisfies $e \leq 1$.
It follows that (MF5) holds.
Thus in the monoid case, the functions  $p \colon I \times I \rightarrow S$  satisfying (MF1)--(MF5)
are precisely what we called {\em normalized, pointed sandwich functions} in \cite{L1}.
Furthermore, the inverse semigroups Morita equivalent to an inverse monoid are precisely the enlargements of that monoid \cite{FLS,L3}.
Thus the theory developed in pages 446--450 of \cite{L1} is the monoid case of the theory we have just developed.}

\item {\em McAlister functions are clearly examples of the manifolds defined by Grandis \cite{G}
and so are related to the approach to sheaves based on Lawvere's paper \cite{Law} and developed by Walters \cite{W}.
See Section~2.8 of \cite{Bor}.}

\item {\em  
The Morita theory of inverse semigroups is initmately connected to the theory of $E$-unitary covers and almost factorizability \cite{L0}.
It has also arisen in the solution of concrete problems \cite{KM}.}

\item {\em In the light of (2) and (3) above, an interesting special case to consider would be where the inverse semigroup is complete and infinitely distributive.}

\end{enumerate}
\end{remarks}


\end{document}